\documentclass[a4,12pt]{article}

\usepackage{amsmath, amssymb, amsthm, multicol}
\usepackage[english]{babel}

\newtheorem{lemma}{Lemma}
\newtheorem{theorem}{Theorem}

\newtheorem{prop}{Proposition}
\newtheorem{coro}{Corollary}

\allowdisplaybreaks

\theoremstyle{remark}
\newtheorem{remark}{Remark}

\def\N{{\mathbb N}} 
\def\Z{{\mathbb Z}}

\def\C{{\mathbb C}}

\def\s{{\bf s}}

\def\x{{\bf x}}

\def\b{{\bf b}}
\def\k{{\bf k}}

\def\Nb{{\bf N}}
\def\P{{\bf P}}

\newcommand{\zerob}{\boldsymbol{0}}

\newcommand{\alphab}{\boldsymbol{\alpha}}

\newcommand{\gammab}{\boldsymbol{\gamma}}
\newcommand{\thetab}{\boldsymbol{\theta}}

\newcommand{\mub}{\boldsymbol{\mu}}
\newcommand{\deltab}{\boldsymbol{\delta}}

\newcommand{\be}{\begin{enumerate}}
\newcommand{\ee}{\end{enumerate}}
\let\ds=\displaystyle

\begin{document}
\title{Values at non-positive integers of partially twisted multiple zeta-functions I}
\author{Driss Essouabri \hskip 0.4cm and \hskip 0.4cm
Kohji Matsumoto}
\date{}
\maketitle 

\noindent
{{\bf {Abstract.}}\\
{\small We study the behavior of partially twisted multiple zeta-functions.
We give new closed and explicit formulas for special values at non-positive integer points of 
such zeta-functions.
Our method is based on a result of M. de Crisenoy on the fully twisted case and the Mellin-Barnes integral formula.
}

\medskip

\noindent
{\small {\bf Mathematics Subject Classifications: Primary 11M32; Secondary 11M41.} }
\setcounter{tocdepth}{2}

\section {Introduction}\label{sec1}

Let $\mathbb{N}$, $\mathbb{N}_0$, $\mathbb{Z}$, $\mathbb{R}$, and $\mathbb{C}$ be the sets of 
positive integers, non-negative integers, 
rational integers, real numbers, and complex numbers, respectively.

Let $\gammab =(\gamma_1,\dots, \gamma_n) \in \C^n$ and  $\b =(b_1,\dots, b_n) \in \C^n$ be two vectors of complex parameters such that 
$\Re (\gamma_j) >0$ and $\Re (b_j) >-\Re (\gamma_1)$ for all $j=1,\dots, n$.
The generalized  Euler-Zagier multiple zeta-function is defined  
for  $n-$tuples of complex variables $\s=(s_1,\dots, s_n)$ by
\begin{equation}\label{EZMzetadef}
\zeta_n(\s; \gammab; \b) :=
\sum_{m_1\geq 1 \atop m_2, \dots, m_n \geq 0} \frac{1}{\prod_{j=1}^n (\gamma_1 m_1+\dots+ \gamma_j m_j +b_j)^{s_j}}.
\end{equation}
This series converges absolutely in the domain 
\begin{equation}\label{domaincv}
\mathcal D_n:=\{\s=(s_1,\dots, s_n)\in \C^n \mid \Re (s_j+\dots+s_n) >n+1-j~~{\rm for \; all}\; j=1,\dots, n\}
\end{equation}
(see \cite{MatIllinois}),
and has the meromorphic continuation to the whole complex space $\C^n$ whose 
possible poles are located in the union of the hyperplanes 
$$s_j+\dots+s_n =(n+1-j)-k_j \quad (1\leq j\leq n,~k_1,\dots, k_n \in \N_0).$$
Moreover it is known that for $n\geq 2$, the points 
$\s=-\Nb$, where $\Nb =(N_1,\dots, N_n)\in \N_0^n$, 
lie in most cases on the singular locus above and are points of indeterminacy.
The evaluation of (limit) values of multiple zeta-functions at those points was first considered by
S. Akiyama, S. Egami and Y. Tanigawa \cite{AET}, and then studied by several subsequent papers 
such as \cite{komori}, \cite{onozuka}, \cite{MOW} and \cite{EM1}.

In \cite{komori}, Y. Komori proved that for any $\Nb =(N_1,\dots, N_n)\in \N_0^n$ and $\thetab=(\theta_1,\dots, \theta_n) \in \C^n$ such that $\theta_j+\dots+\theta_n \neq 0$ for all $j=1,\dots, n$, 
the limit 
\begin{equation}\label{gmzvtheta}
\zeta_n^{\thetab} (-\Nb; \gammab; \b):=\lim_{t\rightarrow 0}\zeta_n(-\Nb+t\thetab; \gammab; \b)
\end{equation} 
exists, and expressed this limit in terms of $\Nb$, $\thetab$ and generalized multiple Bernoulli numbers defined implicitly as coefficients of some multiple series. 

In \cite{EM1}, we
gave a closed explicit formula for $\zeta_n^{\thetab} (-\Nb; \gammab; \b)$ in terms of $\Nb$, $\thetab$ and only classical Bernoulli numbers
$B_k$ $(k\in \N_0)$ defined by
\begin{equation}\label{bernoulli}
\frac{x}{e^x-1}=\sum_{k=0}^{\infty}B_k\frac{x^k}{k!}.
\end{equation}

Moreover in \cite{EM2} we extended partially this result to the case of more general multiple
zeta-functions defined by
\begin{equation}\label{polyzeta}
\zeta_n(\s,\P)=\sum_{m_1,\ldots,m_n\geq 1}\prod_{j=1}^n P_j(m_1,\ldots,m_j)^{-s_j},
\end{equation}
where $\P=(P_1,\ldots,P_n)$ with certain polynomials $P_j\in\mathbb{R}[X_1,\ldots,X_j]$.
In this general case, instead of Bernoulli numbers, certain period integrals appear
in the result.

Now we consider the twisted situation.    
Let $\mathbb{T}=\{z\in\mathbb{C}\mid |z|=1\}$, and
let $\boldsymbol{\mu}_k=(\mu_1,\ldots,\mu_k)\in(\mathbb{T}\setminus\{1\})^k $, 
where $k\in \{0,\dots, n\}$. 

The natural twisted version of \eqref{polyzeta}
is
\begin{equation}\label{twistedpolyzeta}
\zeta_n(\s,\P,\boldsymbol{\mu}_k)=\sum_{m_1,\ldots,m_n\geq 1} 
\frac{\prod_{j=1}^k \mu_j^{m_j}}{\prod_{j=1}^n P_j(m_1,\ldots,m_j)^{s_j}}.
\end{equation}
It follows from the method of \cite{essouabriThesis} (see also \cite{essouabriFourier}) that these series have meromorphic continuation to $\C^n$ for 
fairly general class of polynomials $P_j$.

When $k=n$, that is the ``fully twisted'' case, this type of multiple series 
was studied by M. de Crisenoy \cite{dC}.
Under certain conditions, he proved that
$\zeta_n(\s,\P,\boldsymbol{\mu}_n)$ is entire, so its behavior is much simpler than the
non-twisted case.     He obtained an explicit formula for its values 
at non-positive integer points in terms of Lerch zeta-functions
(see Proposition \ref{prop_dC} below for the exact statement).

The aim of the present series of papers is to consider the case when $k<n$.     
Then $\zeta_n(\s,\P,\boldsymbol{\mu}_k)$ is usually not entire, and
the complexity of its set of singularities and therefore the complexity of its special values increases when $k$ decreases.
Our strategy is to begin with the result of de Crisenoy \cite{dC} in the case $k=n$, 
and first consider the case $k=n-1$ by using the Mellin-Barnes integral formula.
Most of the results presented in this paper are actually restricted in this case.
However we also try to consider the case $k=n-2$.
By the same method it is possible to treat the case $k\leq n-3$ in principle,
but the actual argument will become more and more complicated in practice.

In this paper we mainly study the special case when all $P_j$ are linear
polynomials.    After reviewing the result of de Crisenoy
briefly in the next section, we first state the main results in Section \ref{sec3}.    
In the case of twisted generalized Euler-Zagier multiple zeta-functions 
we will give the
completely explicit formulas (Theorems \ref{main2} and \ref{main3}), whose proofs are shown in Sections \ref{sec4} and \ref{sec5}.
We also prove the explicit formula in
the simplest non-linear situation, the ``power sum'' case (Theorem \ref{main_powercase}), 
which will be proved in Section \ref{sec6}.
The treatment of the general non-linear case is postponed to our
next paper \cite{EMtwisted2}.  
\bigskip

{\bf Acknowledgments}
The authors benefit from the financial support of 
the French-Japanese Project entitled ``Zeta-functions of Several Variables and  Applications" 
(PRC CNRS/JSPS 2015-2016).

\section{Review of de Crisenoy's result}\label{sec2}
Here we recall the result of de Crisenoy \cite{dC}.
Let $P_1,\ldots,P_L,Q\in\mathbb{R}[X_1,\ldots,X_n]$,
$\boldsymbol{\mu}_n=(\mu_1,\ldots,\mu_n)\in(\mathbb{T}\setminus\{1\})^n $,
and $\mathbf{s}=(s_1,\ldots,s_L)\in\mathbb{C}^L$.
We write $\P=(P_1,\ldots,P_L)$. 
He considered the general multiple series of the form
\begin{align}\label{def_general}
Z_n(\mathbf{s},\P,Q,\boldsymbol{\mu}_n)=
\sum_{m_1,\ldots,m_n\geq 1} 
\frac{\left(\prod_{j=1}^n \mu_j^{m_j}\right)Q(m_1,\ldots,m_n)}
{\prod_{\ell=1}^L P_{\ell}(m_1,\ldots,m_n)^{s_{\ell}}}.
\end{align}
He introduced the condition {\it HDF}.    A polynomial $P\in\mathbb{R}[X_1,\ldots,X_n]$
is called {\it HDF} if the following (i) and (ii) hold:

(i) $P(\x)>0$ for all $\x=(x_1,\ldots,x_n)\in [1,\infty)^n$, 

(ii) there exists $\varepsilon_0>0$ such that 
$$
\frac{\partial^{\boldsymbol{\beta}}P}{P}(\x)\ll \prod_{j=1}^n 
x_j^{-\varepsilon_0 \beta_j}
$$
for any $\boldsymbol{\beta}=(\beta_1,\ldots,\beta_n)\in \mathbb{N}_0^n$ and
$\x=(x_1,\ldots,x_n)\in [1,\infty)^n$.
(Or equivalently, if $\beta_j\geq 1$
for some $j\in\{1,\ldots,n\}$, then
$(\partial^{\boldsymbol{\beta}}P/P)(\x)\ll x_j^{-\varepsilon_0}$ for
$\x=(x_1,\ldots,x_n)\in [1,\infty)^n$.)

For any $\boldsymbol{\alpha}=(\alpha_1,\ldots,\alpha_L)\in\mathbb{N}_0^L$, we
define $a_{\k,\boldsymbol{\alpha}}=a_{\k,\boldsymbol{\alpha}}(\P,Q)$ 
as the coefficients of the expansion
$$
Q(X_1,\ldots,X_n)\prod_{\ell=1}^L P_{\ell}(X_1,\ldots,X_n)^{\alpha_{\ell}}
=\sum_{\k\in\mathbb{N}_0^n}
a_{\k,\boldsymbol{\alpha}}X_1^{k_1}\cdots X_n^{k_n}.
$$
Denote by $S(\boldsymbol{\alpha})=S(\boldsymbol{\alpha};\P,Q)$ the finite subset 
of $\mathbb{N}_0^n$ consisting of all 
$\k$ for which $a_{\k,\boldsymbol{\alpha}}\neq 0$.
We write $|\x|=|x_1|+\cdots+|x_n|$.
Then, de Crisenoy proved the following results.

\begin{prop}\label{prop_dC}
{\rm (de Crisenoy \cite{dC})}
Assume that the polynomials $P_1,\ldots,P_L$ satisfy the condition {\it HDF}, and that
$\prod_{\ell=1}^L P_{\ell}(\x)$ tends to $\infty$ as $|\x|\to\infty$, 
$\x\in[1,\infty)^n$.  Then

(i) $Z_n(\mathbf{s},\P,Q,\boldsymbol{\mu}_n)$ can be continued to the whole space
$\mathbb{C}^L$ as an entire function.

(ii) For any $\boldsymbol{\alpha}=(\alpha_1,\ldots,\alpha_L)\in\mathbb{N}_0^L$, we have
\begin{align}\label{formula_dC}
Z_n(-\boldsymbol{\alpha},\P,Q,\boldsymbol{\mu}_n)=\sum_{\k\in S(\boldsymbol{\alpha})}
a_{\k,\boldsymbol{\alpha}}\prod_{j=1}^n \phi_{\mu_j}(-\k_j),
\end{align}
where $\phi_{\mu}(s)=\sum_{m=1}^{\infty}\mu^m m^{-s}$ is the twisted (Lerch)
zeta-function.
\end{prop}

It is to be noted here that, since the point $\s=-\boldsymbol{\alpha}$ is 
a regular point of
$Z_n(\mathbf{s},\P,Q,\boldsymbol{\mu}_n)$ because of the assertion (i), we can evaluate 
the value at $\s=-\boldsymbol{\alpha}$ as a finite definite value in the assertion (ii).

\section{Statement of results}\label{sec3}

Our main aim in this paper is to study the partially twisted multiple zeta-functions whose
denominators are linear forms.
Let $\gammab =(\gamma_1,\dots, \gamma_n) \in \C^n$ and  $\b =(\b_1,\dots, b_n) \in \C^n$ be such that 
$\Re (\gamma_j) >0$ and $\Re (b_j) >-\Re (\gamma_1)$ for all $j=1,\dots, n$.

Let $0\leq k\leq n$.
The partially twisted generalized Euler-Zagier multiple zeta-function is defined formally 
for $n-$tuples of complex variables $\s=(s_1,\dots, s_n)$ by
\begin{align}\label{1-1}
\zeta_{n,k}(\mathbf{s},\boldsymbol{\gamma},\mathbf{b},\boldsymbol{\mu}_k)=
\sum_{\stackrel{m_1\geq 1}{m_2,\ldots,m_n\geq 0}}\frac{\prod_{j=1}^k \mu_j^{m_j}}
{\prod_{j=1}^n (\gamma_1 m_1+\cdots+\gamma_j m_j+b_j)^{s_j}}
\end{align}
(when $k=0$, we understand that the numerator on the right-hand side is 1)
which is absolutely convergent in the domain $\mathcal D_n$ (see \eqref{domaincv}).
The meromorphic continuation and the location of singularities of the function
$\zeta_{n,k}(\mathbf{s},\boldsymbol{\gamma},\mathbf{b},\boldsymbol{\mu}_k)$
are discussed in \cite{FKMT-Bessatsu} (which are partly announced in \cite{FKMT-AJM}).

When $k=n$, this series is a special case of \eqref{def_general} studied by 
de Crisenoy \cite{dC}, whose result implies that 
$\zeta_{n,n}(\mathbf{s},\boldsymbol{\gamma},\mathbf{b},\boldsymbol{\mu}_n)$ 
is entire in $\mathbf{s}$. 

When $k<n$, $\zeta_{n,k}(\mathbf{s},\boldsymbol{\gamma},\mathbf{b},\boldsymbol{\mu}_k)$ has meromorphic continuation to $\C^n$, but is not entire. 
Moreover, the complexity of its set of singularities and therefore the complexity of its special values, increases when $k$ decreases. 
In our previous article \cite{EM1} we handled the case $k=0$ (that is, the non-twisted case) by a method different from that in \cite{dC}.  

In the present paper we develop another approach.
Our following two  
theorems (i.e. Theorem \ref{main2} and Theorem \ref{main3}, proved in Section \ref{sec4} and Section \ref{sec5}, respectively) deal with the cases $k=n-1$ and $k=n-2$. In these cases we use, in addition to de Crisenoy's result (Proposition \ref{prop_dC} above), the Mellin-Barnes formula to determine the set of singularities and the values of 
$\zeta_{n,k}(\mathbf{s},\boldsymbol{\gamma},\mathbf{b},\boldsymbol{\mu}_k)$ at non-positive integers.

We prepare some more notations.
\be
\item For any $a\in \C\setminus (-\N_0)$, let $\ds \zeta(s, a)=\sum_{m=0}^{\infty}(m+a)^{-s}$ be the 
Hurwitz zeta-function
(as for the definition of
$\zeta(s,a)$ for any $a\in \C\setminus (-\N_0)$, see \cite[Lemma 1]{MatJNT});
\item For $\s=(s_1,\ldots,s_n) \in \C^n$, and $k<n$, denote
$\s_k =(s_1,\dots, s_k)$.     Similarly we use the notation
$\gammab_k, \b_k,\mub_k$ etc.  
\item For any $\Nb =(N_1,\dots, N_n) \in \N_0^n$ and any $l\in \Z$, let 
$$\ds \Nb_{n-1}^*(l)=(N_1,\dots, N_{n-2}, N_{n-1}+N_n+l).$$
\item 
For $\alphab=(\alpha_1,\ldots,\alpha_n) \in \N_0^n$ and  $\b=(b_1,\ldots,b_n) \in \C^n$ 
we define the polynomial (in $\b$)  $c_n(\b;\alphab, \k)$ 
(where $\k=(k_1,\ldots,k_n) \in \N_0^n, |\k|\leq |\alphab|)$) as the coefficients of the polynomial 
$\prod_{j=1}^n (\sum_{i=1}^j X_i +b_j)^{\alpha_j}$; that is 
\begin{equation}\label{benoulliexpension}
\prod_{j=1}^n (\sum_{i=1}^j X_i +b_j)^{\alpha_j}
= \sum_{\k \in \N_0^n,\atop  |\k|\leq |\alphab|}
 c_n(\b;\alphab, \k) ~X_1^{k_1} \dots X_n^{k_n}.
 \end{equation}
 \item
 Similarly, for $\gammab =(\gamma_1,\dots, \gamma_n) \in \C^n$, 
 $\widetilde{c}_n(\b;\alphab, \k)$  are defined by
\begin{equation}\label{def-widetilde-c}
\prod_{j=1}^n (\sum_{i=1}^j \gamma_iX_i +b_j)^{\alpha_j}
 = \sum_{\k \in \N_0^n,\atop  |\k|\leq |\alphab|}
 \widetilde{c}_n(\b;\alphab, \k) ~X_1^{k_1} \dots X_n^{k_n}.
 \end{equation}

\ee
 
\begin{remark}\label{rem_cc}
The quantities $c_n(\b;\alphab, \k)$ and $\widetilde{c}_n(\b;\alphab, \k)$ appear also
in \cite{EM1}.   Obviously $\widetilde{c}_n(\b;\alphab, \k)=c_n(\b;\alphab, \k)\gamma_1^{k_1}\cdots\gamma_n^{k_n}$.
\end{remark}
 
\begin{theorem}\label{main2}
Let $n\geq 2$.
Let $\gammab =(\gamma_1,\dots, \gamma_n) \in \C^n$ with $\Re (\gamma_j) >0$
for all $j=1,\dots, n$.
Let $\b =(b_1,\dots, b_n) \in \C^n$, satisfying the conditions
\begin{align}\label{condition_b}
\left\{
\begin{array}{l}
\Re (b_j) > -\Re (\gamma_1)\; ({\rm for\;all}\;j=1,\ldots,n), \\ 
b_n-b_{n-1}\notin (-\infty,0], \quad (b_n-b_{n-1})/\gamma_n\notin(-\infty,0].
\end{array}\right.
\end{align}
(The latter two conditions mean that they are in the principal branch.)
Let $\mub_{n-1}=(\mu_1,\dots, \mu_{n-1}) \in \left(\mathbb{T}\setminus\{1\}\right)^{n-1}$.
Then, the series
$\zeta_{n,n-1}(\mathbf{s},\boldsymbol{\gamma},\mathbf{b},\boldsymbol{\mu}_{n-1})$ has meromorphic 
continuation to the whole space $\C^n$ and its possible poles are located only 
on the hyperplane 
$s_n=1$. 
Furthermore,  for any $\Nb =(N_1,\dots, N_n) \in \N_0^n$, we have
\begin{eqnarray}\label{main2formula}
&&\zeta_{n,n-1}(-\Nb,\boldsymbol{\gamma},\mathbf{b},\boldsymbol{\mu}_{n-1}) \\
&&=
-\frac{1}{N_n +1} \sum_{\k \in \N_0^{n-1} \atop 
|\k| \leq |\Nb_{n-1}^*(1)|} c_{n-1}(\b'_{n-1}; \Nb_{n-1}^*(1), \k) \gamma_n^{-1}  \prod_{j=1}^{n-1} \gamma_j^{k_j}\phi_{\mu_j} (-k_j)\nonumber\\
& & + \sum_{l=0}^{N_n} \binom{N_n}{l}\sum_{\k \in \N_0^{n-1} \atop 
|\k| \leq |\Nb_{n-1}^*(-l)|} c_{n-1}(\b'_{n-1}; \Nb_{n-1}^*(-l), \k) \gamma_n^l \left(\prod_{j=1}^{n-1} \gamma_j^{k_j}\phi_{\mu_j} (-k_j) \right) \nonumber\\
&&\qquad\times\zeta \left(-l; \frac{b_n-b_{n-1}}{\gamma_n}\right),
\nonumber 
\end{eqnarray}
where $\b'_{n-1}=(b_1,b_2',\ldots,b'_{n-1})$ with $b'_j=b_j-(\gamma_2+\cdots+\gamma_j)$
$(2\leq j\leq n-1)$ and
$c_{n-1}(\b'_{n-1}; \Nb_{n-1}^*(-l), \k)$ is defined as in \eqref{benoulliexpension}.
\end{theorem}

\begin{remark}
Moreover, the special values of Hurwitz and Lerch zeta-functions appearing on the
right-hand side can be written down more explicitly.    In fact, it is well-known that
$\zeta(-n, a)=-B_{n+1}(a)/(n+1)$, where $B_{n+1}(a)$ denotes the Bernoulli polynomial
of order $n+1$.     As for $\phi_{\mu}(-n)$, we have
$$
\phi_{\mu}(-n)=\frac{(-1)^n\mu}{1-\mu}\sum_{\ell=0}^n \frac{\ell!S(n,\ell)}
{(\mu-1)^{\ell}},
$$
where $S(n,\ell)$ denotes the Stirling number of the second kind attached to
$(n,\ell)$ (see de Crisenoy \cite[Lemma 5.7]{dC}).
\end{remark}

\begin{theorem}\label{main3}
Let $n\geq 2$.
Let $\gammab =(\gamma_1,\dots, \gamma_n) \in \C^n$ 
with $\Re (\gamma_j) >0$ for all $j=1,\dots, n$, and
and  $\b =(b_1,\dots, b_n) \in \C^n$ satisfying \eqref{condition_b}.
Let $\mub_{n-2}=(\mu_1,\dots, \mu_{n-2}) \in \left(\mathbb{T}\setminus \{1\}\right)^{n-2}$.
Then, the series
$\zeta_{n,n-2}(\mathbf{s},\boldsymbol{\gamma},\mathbf{b},\boldsymbol{\mu}_{n-2})$ has meromorphic continuation to the whole space $\C^n$ and its possible singularities 
are located only on the hyperplanes
$$s_n=1 \quad {\mbox { and  }} \quad s_{n-1}+s_n= k ~~(k\in \Z, k\leq 2).$$
Furthermore,
for any $\Nb =(N_1,\dots, N_n) \in \N_0^n$, as $\deltab=(\delta_1,\dots, \delta_n)$ 
tends to
$(0,\dots, 0)$, we have
\begin{eqnarray*}
\lefteqn{\zeta_{n,n-2}(-\Nb+\deltab,\boldsymbol{\gamma},\mathbf{b},\boldsymbol{\mu}_{n-2})}\\ 
&&=-\frac{1}{N_n+1}\zeta_{n-1,n-2}(-\mathbf{N}_{n-1}^*(1),\boldsymbol{\gamma}_{n-1},\mathbf{b}_{n-1},
\boldsymbol{\mu}_{n-2})\gamma_n^{-1}\\
& &+\sum_{l=0}^{N_n}\binom{N_n}{l}\zeta_{n-1,n-2}(-\mathbf{N}_{n-1}^*(-l),\boldsymbol{\gamma}_{n-1},
\mathbf{b}_{n-1},\boldsymbol{\mu}_{n-2})
 ~\zeta\left(-l,\frac{b_n-b_{n-1}}{\gamma_n}\right)\gamma_n^l \\
&&+\frac{(-1)^{N_{n-1}+1} N_n! N_{n-1}!}{(N_{n-1}+N_n+1)!}~\left(\frac{\delta_n+O(\delta_n^2)}{\delta_{n-1}+\delta_n}\right)\\
&&\times\zeta_{n-2,n-2}(-\mathbf{N}_{n-2}
+\deltab_{n-2},\boldsymbol{\gamma}_{n-2},
\mathbf{b}_{n-2},\boldsymbol{\mu}_{n-2})\\
&&\times\zeta\left(-N_{n-1}-N_n-1,\frac{b_n-b_{n-1}}{\gamma_n}\right)\gamma_{n-1}^{-1}
\gamma_n^{N_{n-1}+N_n+1}\\
&&
+O\left(\max_{1\leq j\leq n}|\delta_j|\right),
\end{eqnarray*}
where, when $n=2$, we understand that $\zeta_{0,0}\equiv 1$.
\end{theorem}

\begin{remark}
When $n=2$, Theorem \ref{main3} gives the result on the (non-twisted) double zeta-function,
which coincides with \cite[Corollary 5.2]{MOW}.
\end{remark}

As a corollary, we obtain the following result: 
\begin{coro}
Assume that the assumptions of Theorem \ref{main3} hold. 
Let $\Nb =(N_1,\dots, N_n) \in \N_0^n$ and $\theta \in \C$.
Then, the limit 
$$\zeta_{n,n-2}^{\theta}(-\Nb,\boldsymbol{\gamma},\mathbf{b},\boldsymbol{\mu}_{n-2}):= \lim_{\deltab \rightarrow \zerob,~ \frac{\delta_n}{\delta_{n-1}+\delta_n}\rightarrow \theta} \zeta_{n,n-2}(-\Nb+\deltab,\boldsymbol{\gamma},\mathbf{b},\boldsymbol{\mu}_{n-2})$$
exists and is given by 
\begin{eqnarray*}
\lefteqn{\zeta_{n,n-2}^{\theta}(-\Nb,\boldsymbol{\gamma},\mathbf{b},\boldsymbol{\mu}_{n-2})}\\
&&= 
-\frac{1}{N_n+1}\zeta_{n-1,n-2}(-\mathbf{N}_{n-1}^*(1),\boldsymbol{\gamma}_{n-1},\mathbf{b}_{n-1},
\boldsymbol{\mu}_{n-2})\gamma_n^{-1}\\
& &+\sum_{l=0}^{N_n}\binom{N_n}{l}\zeta_{n-1,n-2}(-\mathbf{N}_{n-1}^*(-l),\boldsymbol{\gamma}_{n-1},
\mathbf{b}_{n-1},\boldsymbol{\mu}_{n-2})
 ~\zeta\left(-l,\frac{b_n-b_{n-1}}{\gamma_n}\right)\gamma_n^l \\
&&+\frac{(-1)^{N_{n-1}+1} N_n! N_{n-1}!}{(N_{n-1}+N_n+1)!}~
\zeta_{n-2,n-2}(-\mathbf{N}_{n-2},\boldsymbol{\gamma}_{n-2},
\mathbf{b}_{n-2},\boldsymbol{\mu}_{n-2}) ~\theta\\
&&\times\zeta\left(-N_{n-1}-N_n-1,\frac{b_n-b_{n-1}}{\gamma_n}\right)\gamma_{n-1}^{-1}
\gamma_n^{N_{n-1}+N_n+1}.
\end{eqnarray*}
Moreover, on the right-hand side, we may apply Theorem \ref{main2} to the 
$\zeta_{n-1,n-2}$ factors and Proposition \ref{prop_dC} to the $\zeta_{n-2,n-2}$
factor, to obtain a more explicit expression of
$\zeta_{n,n-2}^{\theta}(-\Nb,\boldsymbol{\gamma},\mathbf{b},\boldsymbol{\mu}_{n-2})$.
\end{coro}

\bigskip

The argument to prove Theorem \ref{main3} can be extended to the case $k\leq n-3$, to obtain the same type of 
explicit formulas.     However, for smaller values of $k$, more and more relevant singularities 
will appear, so the description of indeterminacy will be much more complicated.\par


By the method in the present paper, it is possible to study the behavior of 
multiple zeta-functions of
more general form \eqref{twistedpolyzeta}, whose denominators are not
necessarily linear forms.  
The general treatment will be developed in our next paper \cite{EMtwisted2},  
but here, we discuss the following special type of non-linear forms.

Let $\mathbf{h}=(h_1,\ldots,h_n)\in\mathbb{N}^n$, and define
\begin{align}\label{2-1}
\zeta_{n,k}(\mathbf{s},\mathbf{h},\boldsymbol{\gamma},\mathbf{b},\boldsymbol{\mu}_k)=
\sum_{m_1,\ldots,m_n\geq 1}\frac{\prod_{j=1}^k \mu_j^{m_j}}
{\prod_{j=1}^n (\gamma_1 m_1^{h_1}+\cdots+\gamma_j m_j^{h_j}+b_j)^{s_j}}.
\end{align}
Analogous to \eqref{def-widetilde-c}, we define
$\widetilde{c}_n(\b;\mathbf{h},\alphab, \k)$ by
\begin{equation}\label{def-widetilde-c2}
\prod_{j=1}^n (\sum_{i=1}^j \gamma_iX_i^{h_i} +b_j)^{\alpha_j}
 = \sum_{\k \in \N_0^n}
 \widetilde{c}_n(\b;\mathbf{h},\alphab, \k) ~X_1^{k_1} \dots X_n^{k_n}.
 \end{equation}
Note that the sum on the right-hand side is actually a finite sum.
As in Section \ref{sec2}, we denote by $S(\boldsymbol{\alpha})$ the set of all
$\k$ such that $\widetilde{c}_n(\b;\mathbf{h},\alphab, \k)\neq 0$.
Using this notation, we can formulate our third main result as follows.

\begin{theorem}\label{main_powercase}
Under the same assumptions as in Theorem \ref{main2}, we have
\begin{align}\label{formula_powercase}
&\zeta_{n,n-1}(-\mathbf{N},\mathbf{h},\boldsymbol{\gamma},\mathbf{b},\boldsymbol{\mu}_{n-1})\\
&= -\frac{\delta_{1,h_n}}{N_n+1}\sum_{\k\in S(\mathbf{N}_{n-1}^*(1))}
\widetilde{c}_{n-1}(\b_{n-1};\mathbf{h}_{n-1},\mathbf{N}_{n-1}^*(1), \k)
\gamma_n^{-1}\prod_{j=1}^{n-1}\phi_{\mu_j}(-k_j)\notag\\
&+\sum_{l=0}^{N_n}\binom{N_n}{l}
\sum_{\k\in S(\mathbf{N}_{n-1}^*(-l))}
\widetilde{c}_{n-1}(\b_{n-1};\mathbf{h}_{n-1},\mathbf{N}_{n-1}^*(-l), \k)
\gamma_n^{l}\left(\prod_{j=1}^{n-1}\phi_{\mu_j}(-k_j)\right)
\notag\\
&\qquad\times\zeta\left(-l,h_n,\frac{b_n-b_{n-1}}{\gamma_n}\right),\notag
\end{align}
where $\delta_{1,h_n}$ denotes the Kronecker delta.
\end{theorem}

\section{Proof of Theorem \ref{main2}}\label{sec4}

Now we start the proof of Theorem \ref{main2}.

Let $n\geq 2$, and
fix $\gammab =(\gamma_1,\dots, \gamma_n) \in \C^n$ and  $\b =(b_1,\dots, b_n) \in \C^n$  such that 
$\Re (\gamma_j) >0$ and $\Re (b_j) > -\Re (\gamma_1)$ for all $j=1,\dots, n$ and 
$(b_n-b_{n-1})/\gamma_n\notin(-\infty,0]$.
Fix also  $\mub_{n-1}=(\mu_1,\dots, \mu_{n-1}) \in \left(\mathbb{T}\setminus \{1\}\right)^{n-1}$.

The zeta function 
$$\zeta_{n,n-1}(\mathbf{s},\boldsymbol{\gamma},\mathbf{b},\boldsymbol{\mu}_{n-1})=
\sum_{\stackrel{m_1\geq 1}{m_2,\ldots,m_n\geq 0}}\frac{\prod_{j=1}^{n-1} \mu_j^{m_j}}
{\prod_{j=1}^n (\gamma_1 m_1+\cdots+\gamma_j m_j+b_j)^{s_j}}
$$
is absolutely convergent (see \cite{MatIllinois}) in the region
$\mathcal{D}_n$,
hence especially in its subregion
$$
\mathcal{A}_n=\{\mathbf{s}\in\mathbb{C}^n\;|\;
\Re s_j>1 \;(1\leq j\leq n)\}.
$$

Recall the Mellin-Barnes integral formula:
\begin{align}\label{1-2}
(1+\lambda)^{-s}=\frac{1}{2\pi i}\int_{(c)}\frac{\Gamma(s+z)\Gamma(-z)}{\Gamma(s)}
\lambda^z dz,
\end{align}
where $s,\lambda\in\mathbb{C}$, $\Re s>0$, $\lambda\neq 0$, $|\arg\lambda|<\pi$
(the principal branch),
$-\Re s<c<0$, and the path of the integral is the vertical line $\Re z=c$
(see \cite{WW}).

Here we assume temporarily that $\s\in\mathcal{A}_n$ and
\begin{align}\label{1-3} 
\Re(b_n-b_{n-1})>0.
\end{align}
Our starting point is the decomposition
\begin{align}\label{1-4}
&(\gamma_1 m_1+\cdots+\gamma_n m_n+b_n)^{-s_n}\\
&\quad=(\gamma_1 m_1+\cdots+\gamma_{n-1} m_{n-1}+b_{n-1})^{-s_n}\notag\\
&\qquad\times\left(1+\frac{\gamma_n m_n+b_n-b_{n-1}}{\gamma_1 m_1+\cdots+\gamma_{n-1} m_{n-1}+b_{n-1}}
\right)^{-s_n}.\notag
\end{align}
Under the assumption \eqref{1-3} we see that
$$
\left|\arg\left(\frac{\gamma_n m_n+b_n-b_{n-1}}{\gamma_1 m_1+\cdots+\gamma_{n-1} m_{n-1}+b_{n-1}}
\right)\right|<\pi,
$$
hence the above decomposition \eqref{1-4} is valid, and using \eqref{1-2} we obtain
\begin{align}\label{1-5}
&(\gamma_1 m_1+\cdots+\gamma_n m_n+b_n)^{-s_n}\\
&=(\gamma_1 m_1+\cdots+\gamma_{n-1} m_{n-1}+b_{n-1})^{-s_n}\notag\\
&\times
\frac{1}{2\pi i}\int_{(c)}\frac{\Gamma(s_n+z)\Gamma(-z)}{\Gamma(s_n)}
\left(\frac{\gamma_n m_n+b_n-b_{n-1}}{\gamma_1 m_1+\cdots+\gamma_{n-1} m_{n-1}+b_{n-1}}
\right)^z dz,\notag
\end{align}
where $-\Re s_n<c<0$.    But since $\mathbf{s}\in\mathcal{A}_n$, we have
$\Re s_n>1$, so we may assume (more strongly)
\begin{align}\label{1-6}
-\Re s_n<c<-1.
\end{align}
Substituting \eqref{1-5} into \eqref{1-1} (with $k=n-1$) and changing the order of integration 
and summation, we have
\begin{align}\label{1-7}
&\zeta_{n,n-1}(\mathbf{s},\boldsymbol{\gamma},\mathbf{b},\boldsymbol{\mu}_{n-1})\\
&=\frac{1}{2\pi i}\int_{(c)}\frac{\Gamma(s_n+z)\Gamma(-z)}{\Gamma(s_n)}
\sum_{\stackrel{m_1\geq 1}{m_2,\ldots,m_n\geq 0}}\frac{\prod_{l=1}^{n-1} \mu_l^{m_l}}
{\prod_{j=1}^{n-2} (\gamma_1 m_1+\cdots+\gamma_j m_j+b_j)^{s_j}}\notag\\
&\times(\gamma_1 m_1+\cdots+\gamma_{n-1} m_{n-1}+b_{n-1})^{-s_{n-1}-s_n-z}
(\gamma_n m_n+b_n-b_{n-1})^z dz\notag\\
&=\frac{1}{2\pi i}\int_{(c)}\frac{\Gamma(s_n+z)\Gamma(-z)}{\Gamma(s_n)}
\zeta_{n-1,n-1}(\mathbf{s}_{n-1}^*(z),\boldsymbol{\gamma}_{n-1},\mathbf{b}_{n-1},
\boldsymbol{\mu}_{n-1})\notag\\
&\qquad\times \zeta\left(-z,\frac{b_n-b_{n-1}}{\gamma_n}\right)\gamma_n^z dz,\notag
\end{align}
where $\mathbf{s}_{n-1}^*(z)=(s_1,\ldots,s_{n-2},s_{n-1}+s_n+z)$.
(Under the assumption \eqref{1-6}, both of the above two zeta factors in the integrand are convergent.)

Let $M$ be a positive integer, and now we shift the path of integration to $\Re z=M+1/2$.
We claim that {\it this shifting is possible, and also we can remove the 
assumption \eqref{1-3}}.
In fact, in the strip $c\leq \Re z\leq M+1/2$, by Stirling's formula we have 
$$
\Gamma(s_n+z)\Gamma(-z)\ll e^{-\pi(|\Im s_n|/2+|\Im z|)}
(|\Im s_n|+|\Im z|+1)^{\Re s_n+\Re z-1/2}(|\Im z|+1)^{-\Re z-1/2}.
$$
The factor 
$\zeta_{n-1,n-1}(\mathbf{s}_{n-1}^*(z),\boldsymbol{\gamma}_{n-1},\mathbf{b}_{n-1},
\boldsymbol{\mu}_{n-1})$
is $O(1)$ for any $\b_{n-1}$ satisfying \eqref{condition_b},
because it is in the domain of absolute convergence.
As we mentioned in
Section \ref{sec3}, the Hurwitz zeta-function
$\zeta(s,a)$ can be defined for any complex $a$ except for the case when 
$a=-l$, $l\in\mathbb{N}_0$.
Moreover it holds that
\begin{align}\label{Hurwitz_estimate}
\zeta(s,a/w)w^{-s}=O\left(|w|^{-\Re s}(|\Im s|+1)^{\max\{0,1-\Re s\}+\varepsilon}
\exp(|\Im s|\max\{|\arg a|,|\arg w|\})\right)
\end{align}
if $a/w\notin(-\infty,0]$ (see \cite[Lemma 2]{MatJNT}).   Therefore, under the 
assumption
$(b_n-b_{n-1})/\gamma_n\notin(-\infty,0]$,
we have
\begin{align*}
\lefteqn{\zeta\left(-z,\frac{b_n-b_{n-1}}{\gamma_n}\right)\gamma_n^z}\\
&\quad\ll
(|\Im z|+1)^{\max\{0,1+\Re z\}+\varepsilon}\exp(|\Im z|\max\{|\arg(b_n-b_{n-1})|,
|\arg\gamma_n|\}).
\end{align*}
These estimates imply that the integrand on the right-hand side of \eqref{1-7} is
\begin{align*}
&\ll ({\rm The\; factor\; of\; polynomial\; order\; in}\; |\Im z|)\\
&\quad\times\exp(|\Im z|(\max\{|\arg(b_n-b_{n-1})|,|\arg\gamma_n|\}-\pi))
\end{align*}
(here, the implied constant may depend on $s_n$).
Therefore, if we further assume $b_n-b_{n-1}\notin(-\infty,0]$, we see that the integrand
is of exponential decay.    This implies that, only under the assumption
\eqref{condition_b}, the integral is absolutely convergent, and the
indicated shifting of the path of integral is possible.   
The assumption \eqref{1-3} is not necessary (or in other words, we can continue
\eqref{1-7} with respect to $\b$ to the wider region given by \eqref{condition_b}).
The proof of the claim is complete.

Carrying out this shifting, we find that
the relevant poles are $z=-1$ (from the Hurwitz zeta factor)
and $z=0,1,2,\ldots,M$ (from $\Gamma(-z)$).    Counting the residues, we obtain
\begin{align}\label{1-8}
&\zeta_{n,n-1}(\mathbf{s},\boldsymbol{\gamma},\mathbf{b},\boldsymbol{\mu}_{n-1})\\
&=\frac{1}{s_n-1}\zeta_{n-1,n-1}(\mathbf{s}_{n-1}^*(-1),\boldsymbol{\gamma}_{n-1},\mathbf{b}_{n-1},
\boldsymbol{\mu}_{n-1})\gamma_n^{-1}\notag\\
&+\sum_{l=0}^M\binom{-s_n}{l}\zeta_{n-1,n-1}(\mathbf{s}_{n-1}^*(l),\boldsymbol{\gamma}_{n-1},
\mathbf{b}_{n-1},\boldsymbol{\mu}_{n-1})\notag\\
&\qquad\times\zeta\left(-l,\frac{b_n-b_{n-1}}{\gamma_n}\right)\gamma_n^l\notag\\
&+\frac{1}{2\pi i}\int_{(M+1/2)}\frac{\Gamma(s_n+z)\Gamma(-z)}{\Gamma(s_n)}
\zeta_{n-1,n-1}(\mathbf{s}_{n-1}^*(z),\boldsymbol{\gamma}_{n-1},\mathbf{b}_{n-1},
\boldsymbol{\mu}_{n-1})\notag\\
&\qquad\times \zeta\left(-z,\frac{b_n-b_{n-1}}{\gamma_n}\right)\gamma_n^z dz.\notag
\end{align}
Since $\zeta_{n-1,n-1}$ is entire, the poles (in $z$) of the integrand of the above integral
are $z=-1, 0,1,2,\ldots$ and $z=-s_n,-s_n-1,-s_n-2,\ldots$.
Therefore the above integral can be continued holomorphically to the region
satisfying $\Re(-s_n)<M+1/2$, that is, 
$$
\{\mathbf{s}\in\mathbb{C}^n\;|\;\Re s_n>-M-1/2\}.
$$
Since $M$ is arbitrary, we can show from \eqref{1-8} that
$\zeta_{n,n-1}(\mathbf{s},\boldsymbol{\gamma},\mathbf{b},\boldsymbol{\mu}_{n-1})$
can be continued meromorphically to the whole space $\mathbb{C}^n$.    Moreover, again noting
that $\zeta_{n-1,n-1}$ is entire, we find that the only possible singularity is the
hyperplane $s_n=1$.


Let $\mathbf{N}=(N_1,\ldots,N_n)\in\mathbb{N}_0^n$.
Then $\mathbf{s}=-\mathbf{N}$ is a regular point of the function
$\zeta_{n,n-1}(\mathbf{s},\boldsymbol{\gamma}_n,\mathbf{b}_n,\boldsymbol{\mu}_{n-1})$.

Put $\mathbf{s}=-\mathbf{N}$ on \eqref{1-8}.   Then the integral is equal to
0, because of the factor $\Gamma(s_n)$ on the denominator.    Also, when
$l>N_n$, then the binomial coefficient $\binom{N_n}{l}$
is equal to 0.   
(We may assume that $M$ is sufficiently large, satisfying $M>N_n$.)
Noting $\s_{n-1}^*(l)|_{\s=-\mathbf{N}}=-\mathbf{N}_{n-1}^*(-l)$, we obtain
the following explicit formula:
\begin{align}\label{1-10}
&\zeta_{n,n-1}(-\mathbf{N},\boldsymbol{\gamma},\mathbf{b},\boldsymbol{\mu}_{n-1})\\
&=-\frac{1}{N_n+1}\zeta_{n-1,n-1}(-\mathbf{N}_{n-1}^*(1),\boldsymbol{\gamma}_{n-1},\mathbf{b}_{n-1},
\boldsymbol{\mu}_{n-1})\gamma_n^{-1}\notag\\
&+\sum_{l=0}^{N_n}\binom{N_n}{l}\zeta_{n-1,n-1}(-\mathbf{N}_{n-1}^*(-l),\boldsymbol{\gamma}_{n-1},
\mathbf{b}_{n-1},\boldsymbol{\mu}_{n-1})\notag\\
&\qquad\times\zeta\left(-l,\frac{b_n-b_{n-1}}{\gamma_n}\right)\gamma_n^l.\notag
\end{align}
The special values 
$\zeta_{n-1,n-1}(-\k_{n-1},\boldsymbol{\gamma}_{n-1},\mathbf{b}_{n-1},\boldsymbol{\mu}_{n-1})$
(where $\k_{n-1}=(k_1,\dots, k_{n-1}) \in \N_0^{n-1}$) are evaluated explicitly by  
Proposition \ref{prop_dC}
in terms of special values of the Lerch zeta-function
$\phi_{\mu_j}(s)$.
Since
$$-N_{n-1}-N_n+l\leq -N_{n-1}\leq 0$$ 
for $l\leq N_n$, we can apply Proposition \ref{prop_dC} to the factors $\zeta_{n-1,n-1}$ 
appearing on the right-hand side of the above.    

Let $b_1'=b_1$, $b_j'=b_j-(\gamma_2+\cdots+\gamma_j)$ ($2\leq j\leq n-1$).   Then we can write
$$
\zeta_{n-1,n-1}(\s_{n-1},\boldsymbol{\gamma}_{n-1},\mathbf{b}_{n-1},
\boldsymbol{\mu}_{n-1})
=\sum_{m_1,\ldots,m_{n-1}\geq 1}\frac{\prod_{j=1}^{n-1}\mu_j^{m_j}}
{\prod_{j=1}^{n-1}(\gamma_1 m_1+\cdots+\gamma_j m_j+b'_j)^{s_j}},
$$
which agrees with the notation of Proposition \ref{prop_dC}.    
Since Proposition \ref{prop_dC} is proved for polynomials of real coefficients, 
here we temporarily assume that $\gamma_j, b_j\in\mathbb{R}$ ($1\leq j\leq n$).
Then the {\it HDF} condition is clearly satisfied, and
by Proposition \ref{prop_dC} we have
\begin{eqnarray*}
\lefteqn{\zeta_{n-1,n-1}(-\mathbf{N}_{n-1}^*(-l),\boldsymbol{\gamma}_{n-1},\mathbf{b}_{n-1},
\boldsymbol{\mu}_{n-1})}\\
&&=\sum_{\k \in \N_0^{n-1} \atop 
|\k| \leq |\Nb_{n-1}^*(-l)|} \widetilde{c}_{n-1}(\b'_{n-1}; \Nb_{n-1}^*(-l), \k)  
\prod_{j=1}^{n-1}\phi_{\mu_j} (-k_j) \quad (l\geq -1),
\end{eqnarray*}
where $\widetilde{c}_{n-1}(\b'_{n-1}; \Nb_{n-1}^*(-l), \k)$ is that defined by
\eqref{def-widetilde-c}.
Applying this to the right-hand side of \eqref{1-10}, and noting Remark \ref{rem_cc},
we obtain the assertion of Theorem \ref{main2}. 
The restriction $\gamma_j, b_j\in\mathbb{R}$ can be removed by the analytic continuation
with respect to $\gamma_j, b_j$.
\qed

\begin{remark}
For $\mub \in \left(\mathbb{T}\setminus \{1\}\right)^n$, we can apply the same argument as above to
$\zeta_{n,n}(\mathbf{s},\boldsymbol{\gamma},\mathbf{b},\boldsymbol{\mu})$.
The result is that the special values
$\zeta_{n,n}(-\mathbf{N},\boldsymbol{\gamma},\mathbf{b},\boldsymbol{\mu})$
can be written in terms of
$\zeta_{n-1,n-1}(-\mathbf{N}_{n-1}^*(-l),\boldsymbol{\gamma}_{n-1},
\mathbf{b}_{n-1},\boldsymbol{\mu}_{n-1})$
and special values of the Hurwitz-Lerch zeta-function
$$
\phi\left(s,\frac{b_n-b_{n-1}}{\gamma_n},\mu_n\right)=\sum_{m=0}^{\infty}
\mu_n^m\left(m+\frac{b_n-b_{n-1}}{\gamma_n}\right)^{-s}.
$$
This gives another way of computing the special values by induction.
\end{remark}

\section{Proof of Theorem \ref{main3}}\label{sec5}

First, as a preparation, we consider the behavior of $\zeta_{n,n-1}$ around its
singularity $s_n=1$.   We will use in the sequel of this section the notations 
of Section \ref{sec4}.

Let $n\geq 2$.
Let $s_n=1+\delta_n$, where $\delta_n$ is a small (non-zero) complex number.
Then
\begin{align*}
&\frac{1}{s_n-1}\zeta_{n-1,n-1}(\mathbf{s}_{n-1}^*(-1),\boldsymbol{\gamma}_{n-1},\mathbf{b}_{n-1},
\boldsymbol{\mu}_{n-1})\gamma_n^{-1}\\
&=\frac{1}{\delta_n}\zeta_{n-1,n-1}((s_1,\ldots,s_{n-2},s_{n-1}+\delta_n),\boldsymbol{\gamma}_{n-1},\mathbf{b}_{n-1},
\boldsymbol{\mu}_{n-1})\gamma_n^{-1}\\
&=\frac{1}{\delta_n}\zeta_{n-1,n-1}(\mathbf{s}_{n-1},\boldsymbol{\gamma}_{n-1},\mathbf{b}_{n-1},
\boldsymbol{\mu}_{n-1})\gamma_n^{-1}\\
&+\frac{\partial}{\partial s_{n-1}}\zeta_{n-1,n-1}(\mathbf{s}_{n-1},\boldsymbol{\gamma}_{n-1},
\mathbf{b}_{n-1},\boldsymbol{\mu}_{n-1})\gamma_n^{-1}+O(|\delta_n|),
\end{align*}
so from \eqref{1-8} we have
\begin{align}\label{1-8bis}
&\zeta_{n,n-1}((s_1,\ldots,s_{n-1},1+\delta_n),\boldsymbol{\gamma}_n,\mathbf{b}_n,\boldsymbol{\mu}_{n-1})\\
&=\frac{1}{\delta_n}\zeta_{n-1,n-1}(\mathbf{s}_{n-1},\boldsymbol{\gamma}_{n-1},\mathbf{b}_{n-1},
\boldsymbol{\mu}_{n-1})\gamma_n^{-1}\notag\\
&+\frac{\partial}{\partial s_{n-1}}\zeta_{n-1,n-1}(\mathbf{s}_{n-1},\boldsymbol{\gamma}_{n-1},
\mathbf{b}_{n-1},\boldsymbol{\mu}_{n-1})\gamma_n^{-1}\notag\\
&+\sum_{l=0}^M\binom{-1}{l}\zeta_{n-1,n-1}((s_1,\ldots,s_{n-2},s_{n-1}+1+l),\boldsymbol{\gamma}_{n-1},
\mathbf{b}_{n-1},\boldsymbol{\mu}_{n-1})\notag\\
&\qquad\times\zeta\left(-l,\frac{b_n-b_{n-1}}{\gamma_n}\right)\gamma_n^l\notag\\
&+\frac{1}{2\pi i}\int_{(M+1/2)}\Gamma(1+z)\Gamma(-z)
\zeta_{n-1,n-1}((s_1,\ldots,s_{n-2},s_{n-1}+1+z),\notag\\
&\qquad\boldsymbol{\gamma}_{n-1},\mathbf{b}_{n-1},
\boldsymbol{\mu}_{n-1})\zeta\left(-z,\frac{b_n-b_{n-1}}{\gamma_n}\right)\gamma_n^z dz
+O(|\delta_n|)\notag\\
&=\frac{1}{\delta_n}\zeta_{n-1,n-1}(\mathbf{s}_{n-1},\boldsymbol{\gamma}_{n-1},\mathbf{b}_{n-1},
\boldsymbol{\mu}_{n-1})\gamma_n^{-1}\notag\\
&+B(\mathbf{s}_{n-1},\boldsymbol{\gamma}_{n-1},\mathbf{b}_{n-1},
\boldsymbol{\mu}_{n-1})+O(|\delta_n|),\notag
\end{align}
say.   
So far we have worked under the assumption $n\geq 2$.   However when $n=1$, we see that
\begin{align*}
&\zeta_{1,0}(1+\delta_1,\boldsymbol{\gamma}_{1},\mathbf{b}_{1},
\boldsymbol{\mu}_{0})
=\sum_{m_1=1}^{\infty}(\gamma_1 m_1+b_1)^{-1-\delta_1}\\
&=\gamma_1^{-1-\delta_1}\zeta(1+\delta_1,b_1/\gamma_1)-b_1^{-1-\delta_1}
=\frac{1}{\delta_1}\gamma_1^{-1}+({\rm constant})+O(|\delta_1|),
\end{align*}
so \eqref{1-8bis} is valid also for $n=1$ with the convention $\zeta_{0,0}=1$.

Now we start the proof of Theorem \ref{main3}.   Let $n\geq 2$.
Fix $\gammab =(\gamma_1,\dots, \gamma_n) \in \C^n$ and  $\b =(b_1,\dots, b_n) \in \C^n$  such that 
$\Re (\gamma_j) >0$ and $\Re (b_j) > -\Re (\gamma_1)$ for all $j=1,\dots, n$ and 
$(b_n-b_{n-1})/\gamma_n \notin(-\infty,0]$.
Fix also  $\mub_{n-2}=(\mu_1,\dots, \mu_{n-2}) \in \left(\mathbb{T}\setminus \{1\}\right)^{n-2}$.

Assume $\mathbf{s}\in\mathcal{A}_n$.
Analogous to \eqref{1-7}, this time we obtain
\begin{align}\label{1-11}
&\zeta_{n,n-2}(\mathbf{s},\boldsymbol{\gamma},\mathbf{b},\boldsymbol{\mu}_{n-2})\\
&=\frac{1}{2\pi i}\int_{(c)}\frac{\Gamma(s_n+z)\Gamma(-z)}{\Gamma(s_n)}
\zeta_{n-1,n-2}(\mathbf{s}_{n-1}^*(z),\boldsymbol{\gamma}_{n-1},\mathbf{b}_{n-1},
\boldsymbol{\mu}_{n-2})\notag\\
&\qquad\times \zeta\left(-z,\frac{b_n-b_{n-1}}{\gamma_n}\right)\gamma_n^z dz,\notag
\end{align}
where $-\Re s_n<c<-1$.
The factor 
$$\zeta_{n-1,n-2}(\mathbf{s}_{n-1}^*(z),\boldsymbol{\gamma}_{n-1},\mathbf{b}_{n-1},
\boldsymbol{\mu}_{n-2})$$
is not entire, but its pole $s_{n-1}+s_n+z=1$, that is, $z=1-s_{n-1}-s_n$  is irrelevant when
we shift the path from $\Re z=c$ to $\Re z=M+1/2$, because
$\Re(1-s_{n-1}-s_n)<-\Re s_n<c$.
Therefore, analogous to \eqref{1-8}, we have
\begin{align}\label{1-12}
&\zeta_{n,n-2}(\mathbf{s},\boldsymbol{\gamma},\mathbf{b},\boldsymbol{\mu}_{n-2})\\
&=\frac{1}{s_n-1}\zeta_{n-1,n-2}(\mathbf{s}_{n-1}^*(-1),\boldsymbol{\gamma}_{n-1},\mathbf{b}_{n-1},
\boldsymbol{\mu}_{n-2})\gamma_n^{-1}\notag\\
&+\sum_{l=0}^M\binom{-s_n}{l}\zeta_{n-1,n-2}(\mathbf{s}_{n-1}^*(l),\boldsymbol{\gamma}_{n-1},
\mathbf{b}_{n-1},\boldsymbol{\mu}_{n-2})\notag\\
&\qquad\times\zeta\left(-l,\frac{b_n-b_{n-1}}{\gamma_n}\right)\gamma_n^l\notag\\
&+\frac{1}{2\pi i}\int_{(M+1/2)}\frac{\Gamma(s_n+z)\Gamma(-z)}{\Gamma(s_n)}
\zeta_{n-1,n-2}(\mathbf{s}_{n-1}^*(z),\boldsymbol{\gamma}_{n-1},\mathbf{b}_{n-1},
\boldsymbol{\mu}_{n-2})\notag\\
&\qquad\times \zeta\left(-z,\frac{b_n-b_{n-1}}{\gamma_n}\right)\gamma_n^z dz.\notag
\end{align}
Here, the (unique) singularity of 
$\zeta_{n-1,n-2}(\mathbf{s}_{n-1}^*(l),\boldsymbol{\gamma}_{n-1},
\mathbf{b}_{n-1},\boldsymbol{\mu}_{n-2})$
is $s_{n-1}+s_n=1-l$ ($l=-1,0,1,2,\ldots,M$).
Letting $M\to\infty$ we obtain the meromorphic continuation of 
$\zeta_{n,n-2}(\mathbf{s},\boldsymbol{\gamma},\mathbf{b},\boldsymbol{\mu}_{n-2})$, 
and its (possible) singularities are
\begin{align}\label{1-13}
\left\{
  \begin{array}l
    s_n=1,\\
    s_{n-1}+s_n=2,1,0,-1,-2,\ldots
  \end{array}\right.
\end{align}

Now we want to evaluate the value of 
$\zeta_{n,n-2}(\mathbf{s},\boldsymbol{\gamma},\mathbf{b},\boldsymbol{\mu}_{n-2})$
at $\mathbf{s}=-\mathbf{N}\in -\mathbb{N}_0^n$.   
The above \eqref{1-13} shows that $\mathbf{s}=-\mathbf{N}$ can be on a singular locus.

Let $\boldsymbol{\delta}=(\delta_1,\ldots,\delta_n)$, 
where $\delta_j$s are small (non-zero) complex numbers, and
observe the right-hand side of
\eqref{1-12} with $\mathbf{s}=-\mathbf{N}+\boldsymbol{\delta}$.    Since
$$
-\mathbf{N}_{n-1}^*(-l)=(-N_1,\ldots,-N_{n-2},-N_{n-1}-N_n+l), 
$$
the only relevant singularity of
$\zeta_{n-1,n-2}$ factor appears when $l=  N_{n-1}+N_n+1$.
(We may assume $M>N_{n-1}+N_n+1$.)
Analogous to \eqref{1-10}, we have
\begin{align}\label{1-14}
&\zeta_{n,n-2}(-\mathbf{N}+\boldsymbol{\delta},\boldsymbol{\gamma},\mathbf{b},\boldsymbol{\mu}_{n-2})\\
&=-\frac{1}{N_n+1}\zeta_{n-1,n-2}(-\mathbf{N}_{n-1}^*(1),\boldsymbol{\gamma}_{n-1},\mathbf{b}_{n-1},
\boldsymbol{\mu}_{n-2})\gamma_n^{-1}\notag\\
&+\sum_{l=0}^{N_n}\binom{N_n}{l}\zeta_{n-1,n-2}(-\mathbf{N}_{n-1}^*(-l),\boldsymbol{\gamma}_{n-1},
\mathbf{b}_{n-1},\boldsymbol{\mu}_{n-2})\notag\\
&\qquad\times\zeta\left(-l,\frac{b_n-b_{n-1}}{\gamma_n}\right)\gamma_n^l
+R(\boldsymbol{\delta})
+O\left(\max_{1\leq j\leq n}|\delta_j|\right),\notag
\end{align}
where $R(\boldsymbol{\delta})$ denotes the contribution coming from the term $l=N_{n-1}+N_n+1$.
Using \eqref{1-8bis}, we can evaluate $R(\boldsymbol{\delta})$ as follows:
\begin{align}\label{1-15}
&R(\boldsymbol{\delta})=\binom{N_n-\delta_n}{N_{n-1}+N_n+1}\\
&\times\zeta_{n-1,n-2}((-N_1+\delta_1,\ldots,-N_{n-2}+\delta_{n-2},1+\delta_{n-1}+\delta_n),\boldsymbol{\gamma}_{n-1},
\mathbf{b}_{n-1},\boldsymbol{\mu}_{n-2})\notag\\
&\times\zeta\left(-N_{n-1}-N_n-1,\frac{b_n-b_{n-1}}{\gamma_n}\right)
\gamma_n^{N_{n-1}+N_n+1}\notag\\
&=\frac{(N_n-\delta_n)(N_n-1-\delta_n)\cdots
(-\delta_n)\cdots(-N_{n-1}-\delta_n)}{(N_{n-1}+N_n+1)!}\notag\\
&\times\biggl\{\frac{1}{\delta_{n-1}+\delta_n}\zeta_{n-2,n-2}(-\mathbf{N}_{n-2}
+\boldsymbol{\delta}_{n-2},\boldsymbol{\gamma}_{n-2},
\mathbf{b}_{n-2},\boldsymbol{\mu}_{n-2})\gamma_{n-1}^{-1}\biggr.\notag\\
&\biggl.+B(-\mathbf{N}_{n-2}
+\boldsymbol{\delta}_{n-2},\boldsymbol{\gamma}_{n-2},
\mathbf{b}_{n-2},\boldsymbol{\mu}_{n-2})+O(|\delta_{n-1}+\delta_n|)\biggr\}\notag\\
&\times\zeta\left(-N_{n-1}-N_n-1,\frac{b_n-b_{n-1}}{\gamma_n}\right)
\gamma_n^{N_{n-1}+N_n+1}\notag\\
&=\frac{(N_n-\delta_n)(N_n-1-\delta_n)\cdots
(-\delta_n)\cdots(-N_{n-1}-\delta_n)}{(N_{n-1}+N_n+1)!(\delta_{n-1}+\delta_n)}\notag\\
&\times\zeta_{n-2,n-2}(-\mathbf{N}_{n-2}
+\boldsymbol{\delta}_{n-2},\boldsymbol{\gamma}_{n-2},
\mathbf{b}_{n-2},\boldsymbol{\mu}_{n-2})\notag\\
&\times\zeta\left(-N_{n-1}-N_n-1,\frac{b_n-b_{n-1}}{\gamma_n}\right)\gamma_{n-1}^{-1}
\gamma_n^{N_{n-1}+N_n+1}\notag\\
&+O(|\delta_n|),\notag
\end{align}
where $B(\cdot)$ is defined in \eqref{1-8bis}.
This formula describes the situation of indeterminacy.    
We may understand the behavior of
$\zeta_{n,n-2}$ around the point $\mathbf{s}=-\mathbf{N}$ from \eqref{1-14} and \eqref{1-15}.
This ends the proof of Theorem \ref{main3}. \qed

\section{The power sum case}\label{sec6}

In this final section we prove Theorem \ref{main_powercase}.
The series 
$\zeta_{n,k}(\mathbf{s},\mathbf{h},\boldsymbol{\gamma},\mathbf{b},\boldsymbol{\mu}_k)$,
defined by \eqref{2-1},
is an obvious generalization of
$\zeta_{n,k}(\mathbf{s},\boldsymbol{\gamma},\mathbf{b},\boldsymbol{\mu}_k)$
(with a slight change of the condition of the summation), and hence it can be
treated in a quite similar way as in the linear case.

First assume $\s\in\mathcal{A}_n$.
The analogue of \eqref{1-7} is
\begin{align}\label{2-2}
&\zeta_{n,n-1}(\mathbf{s},\mathbf{h},\boldsymbol{\gamma},\mathbf{b},\boldsymbol{\mu}_{n-1})\\
&=\frac{1}{2\pi i}\int_{(c)}\frac{\Gamma(s_n+z)\Gamma(-z)}{\Gamma(s_n)}
\zeta_{n-1,n-1}(\mathbf{s}_{n-1}^*(z),\mathbf{h}_{n-1},\boldsymbol{\gamma}_{n-1},\mathbf{b}_{n-1},
\boldsymbol{\mu}_{n-1})\notag\\
&\qquad\times \zeta\left(-z,h_n,\frac{b_n-b_{n-1}}{\gamma_n}\right)\gamma_n^z dz,\notag
\end{align}
where $-\Re s_n<c<-1$ and
\begin{align}\label{2-3}
\zeta(s,h,b)=\sum_{m=0}^{\infty}\frac{1}{(m^h+b)^s}
\qquad(h\in\mathbb{N}, b\in\mathbb{C}, |\arg b|<\pi).
\end{align}

The analytic properties of $\zeta(s,h,b)$ can also be studied by using the Mellin-Barnes
formula.    

\begin{lemma}\label{lem_h}
The series $\zeta(s,h,b)$ can be continued meromorphically to the whole complex plane.
When $h=1$ (the case of the Hurwitz zeta-function), it has only one pole at $s=1$, 
while when $h\geq 2$, it has infinitely many poles
$s=-l+h^{-1}$ $(l\in\mathbb{N}_0)$.
\end{lemma}

\begin{proof}
First assume $\Re s>1$.
Using the Mellin-Barnes formula \eqref{1-2} we have
\begin{align}\label{2-4}
&\zeta(s,h,b)=b^{-s}+\sum_{m=1}^{\infty}m^{-hs}(1+b/m^h)^{-s}\\
&=b^{-s}+\frac{1}{2\pi i}\sum_{m=1}^{\infty}m^{-hs}\int_{(c_1)}\frac{\Gamma(s+z)\Gamma(-z)}{\Gamma(s)}
\left(\frac{b}{m^h}\right)^z dz\notag
\end{align}
($-\Re s<c_1<0$), which is, after changing the order of integration and summation,
\begin{align}\label{2-5}
=b^{-s}+\frac{1}{2\pi i}\int_{(c_1)}\frac{\Gamma(s+z)\Gamma(-z)}{\Gamma(s)}
\zeta(h(s+z))b^z dz.
\end{align}
To assure the convergence of $\zeta(h(s+z))$, we have to choose $c_1$ satisfying
$h^{-1}-\Re s<c_1<0$.
Now, shift the path to $\Re z=M+1/2$ (which is possible because
$|\arg b|<\pi$), and count the residues of relevant poles at
$z=0,1,2,\ldots,M$.   We obtain
\begin{align}\label{2-6}
\zeta(s,h,b)&=b^{-s}+\sum_{l=0}^M \binom{-s}{l}\zeta(h(s+l))b^l\\
&+\frac{1}{2\pi i}\int_{(M+1/2)}\frac{\Gamma(s+z)\Gamma(-z)}{\Gamma(s)}
\zeta(h(s+z))b^z dz.\notag
\end{align}
Considering the situation $M\to\infty$, we find that \eqref{2-6} gives the meromorphic
continuation of $\zeta(s,h,b)$ to the whole plane.   The Riemann zeta factor 
in the sum on the
right-hand side gives the poles (of order at most 1) at $s=-l+h^{-1}$ ($l\in\mathbb{N}_0$).

When $h=1$, the poles $s=-l+h^{-1}=-l+1$ are cancelled with the binomial factor for $l\geq 1$, so the only pole is $s=1$.    This is of course the case of the Hurwitz zeta-function.
When $h\geq 2$, all of $s=-l+h^{-1}$ ($l\in\mathbb{N}_0$) are really poles.
The residue at $s=-l+h^{-1}$ is
\begin{align}\label{2-7}
\frac{1}{h}\binom{l-h^{-1}}{l}b^l.
\end{align}
\end{proof}

Now, using \eqref{2-6} we evaluate $\zeta(s,h,b)$ for any fixed $s\in\mathbb{C}$.  
Choose $M$ so large as $\Re (h(s+z))>1$ for $\Re z=M+1/2$.  
Denote the integral on the right-hand side of \eqref{2-6} by $J(M)$. 
Then, putting $s=\sigma+it$ and $z=M+1/2+iy$, we see that 
\begin{align*}
J(M)\ll& e^{\pi|t|/2}(|t|+1)^{1/2-\sigma}\int_{-\infty}^{\infty}e^{-\pi|t+y|/2}
(|t+y|+1)^{\sigma+M}\\
&\qquad\times e^{-\pi|y|/2}(|y|+1)^{-M-1}|b|^{M+1/2}e^{|y\arg b|}dy\\
&=e^{\pi|t|/2}(|t|+1)^{1/2-\sigma}|b|^{M+1/2}
 \int_{-\infty}^{\infty}(|t+y|+1)^{\sigma+M}\times (|y|+1)^{-M-1}\\
&\qquad\times\exp\left(-(\pi/2)|t+y|+(|\arg b|-\pi/2)|y|\right)dy.
\end{align*}
Denote the integral on the right-hand side here by $J_1(M)$, and
apply \cite[Lemma 4]{MatJNT} to evaluate $J_1(M)$.     We find that 
\begin{align*}
J_1(M)\ll& (1+(|t|+1)^{\sigma+M})(|t|+1)^{-M-1+\delta(b)}e^{(|\arg b|-\pi/2)|t|}\\
&\qquad + (1+(|t|+1)^{\sigma+M})e^{-\pi|t|/2},
\end{align*} 
where $\delta(b)=1$ if $\arg b=0$ (that is, $b\in\mathbb{R}_{>0}$)
and $\delta(b)=0$ otherwise.
Therefore we have
\begin{align}\label{J(M)}
J(M)\ll |b|^{M+1/2}A_1(|t|)e^{|t\arg b|},
\end{align}
where $A_1(|t|)$ (and $A_2(|t|)$, $A_3(|t|)$ hereafter)
denotes a certain quantity which is of polynomial order in $|t|$.
Therefore from \eqref{2-6} we find that 
\begin{align}\label{zeta_eval}
\zeta(s,h,b)\ll |b|^{\max\{-\sigma,M+1/2\}}A_2(|t|)e^{|t\arg b|}.
\end{align}
(Note that $A_1(|t|)$, $A_2(|t|)$ can be explicitly determined.)
In particular, $\zeta(s,h,1)$ is of polynomial order in $|t|$.
We use this fact to prove the following lemma.

\begin{lemma}\label{lem_est}
Let $s$ be in a fixed vertical strip in $\mathbb{C}$, excluding a small
neighborhood of $s=1$.    If $a,w\in\mathbb{C}$ with $a/w\notin(-\infty,0]$,
then
$$
\zeta(s,h,a/w)w^{-s}=O\left(|w|^{-\sigma}A_3(|t|)\exp(|t|\max\{|\arg a|,|\arg w|\}
)\right)
$$
(the implied constant may depend on $a/w$).
\end{lemma}

\begin{proof}
This lemma is an analogue of \cite[Lemma 2]{MatJNT}, and the proof is similar,
so we just give a brief sketch.
Let $s\in\mathbb{C}$, and we choose $N$ so large that $\Re(s+N)>1$.
As generalizations of \cite[(2.6), (2.10)]{MatJNT}, we can show
\begin{align}\label{lem2a}
\zeta(s,h,b)=\sum_{n=0}^{N-1}\frac{(1-b)^n}{n!}(s)_n \zeta(s+n,h,1)
-(s)_N\int_1^b \frac{(\xi-b)^{N-1}}{(N-1)!}\zeta(s+N,h,\xi)d\xi
\end{align}
(where $(s)_n$ denotes the Pochhammer symbol) and
\begin{align}\label{lem2b}
\zeta(s+N,h,\xi)w^{-s}\ll |w|^{-\sigma}\exp\left(|t|\max\{|\arg a|,|\arg w|\}\right) 
\end{align}
(for any $\xi$ on the segment joining 1 and $a/w$)
by the same argument.    
Putting $b=a/w$ in \eqref{lem2a}, and applying \eqref{lem2b} and the fact mentioned
just before the statement of the lemma, we obtain the assertion. 
\end{proof}

Now let us go back to \eqref{2-2}, and shift the path to $\Re z=M+1/3$.
Here, if we choose $\Re z=M+1/2$ as before, there appears a small problem when
$h_n=2$, so we choose $\Re z=M+1/3$.
The above Lemma \ref{lem_est} ensures that this shifting is possible
(similar to the argument in Section \ref{sec4}).
The relevant poles are $z=0,1,2,\ldots,M$ (from $\Gamma(-z)$)
and $z=-1$ (if $h_n=1$) or $z=l-h_n^{-1}$ ($0\leq l\leq M$, if $h_n\geq 2$).
Analogous to \eqref{1-8}, for $h_n\geq 2$, we obtain
\begin{align}\label{2-8}
&\zeta_{n,n-1}(\mathbf{s},\mathbf{h},\boldsymbol{\gamma},\mathbf{b},\boldsymbol{\mu}_{n-1})\\
&=\sum_{l=0}^M \frac{\Gamma(s_n+l-h_n^{-1})\Gamma(-l+h_n^{-1})}{\Gamma(s_n)}\notag\\
&\qquad\qquad\times\zeta_{n-1,n-1}(\mathbf{s}_{n-1}^*(l-h_n^{-1}),\mathbf{h}_{n-1},\boldsymbol{\gamma}_{n-1},
\mathbf{b}_{n-1},\boldsymbol{\mu}_{n-1})\notag\\
&\qquad\qquad\times \frac{1}{h_n}\binom{l-h_n^{-1}}{l}\left(\frac{b_n-b_{n-1}}{\gamma_n}\right)^l
\gamma_n^{l-h_n^{-1}}\notag\\
&+\sum_{l=0}^M\binom{-s_n}{l}\zeta_{n-1,n-1}(\mathbf{s}_{n-1}^*(l),
\mathbf{h}_{n-1},\boldsymbol{\gamma}_{n-1},
\mathbf{b}_{n-1},\boldsymbol{\mu}_{n-1})\notag\\
&\qquad\times\zeta\left(-l,h_n,\frac{b_n-b_{n-1}}{\gamma_n}\right)\gamma_n^l\notag\\
&+\frac{1}{2\pi i}\int_{(M+1/3)}\frac{\Gamma(s_n+z)\Gamma(-z)}{\Gamma(s_n)}
\zeta_{n-1,n-1}(\mathbf{s}_{n-1}^*(z),\mathbf{h}_{n-1},\boldsymbol{\gamma}_{n-1},\mathbf{b}_{n-1},
\boldsymbol{\mu}_{n-1})\notag\\
&\qquad\times \zeta\left(-z,h_n,\frac{b_n-b_{n-1}}{\gamma_n}\right)\gamma_n^z dz.\notag
\end{align}
If $h_n=1$, then only the term corresponding to $l=0$ on the first sum appears, which is
equal to
\begin{align}\label{onlytheterm}
\frac{1}{s_n-1}\zeta_{n-1,n-1}(\mathbf{s}_{n-1}^*(-1),\mathbf{h}_{n-1},\boldsymbol{\gamma}_{n-1},
\mathbf{b}_{n-1},\boldsymbol{\mu}_{n-1})\gamma_n^{-1}.
\end{align}

Now put $\mathbf{s}=-\mathbf{N}$ and obtain an explicit formula, similar to
\eqref{1-10}.
Because of the existence of the factor $\Gamma(s_n)$ on the denominator, the
integral term vanishes.    If $h_n\geq 2$, the first sum also vanishes by the
same reason.    Therefore
\begin{align}\label{explicit_powercase}
&\zeta_{n,n-1}(-\mathbf{N},\mathbf{h},\boldsymbol{\gamma},\mathbf{b},\boldsymbol{\mu}_{n-1})\\
&= -\frac{\delta_{1,h_n}}{N_n+1}\zeta_{n-1,n-1}(-\mathbf{N}_{n-1}^*(1),\mathbf{h}_{n-1},\boldsymbol{\gamma}_{n-1},
\mathbf{b}_{n-1},\boldsymbol{\mu}_{n-1})\gamma_n^{-1}\notag\\
&+\sum_{l=0}^{N_n}\binom{N_n}{l}\zeta_{n-1,n-1}(-\mathbf{N}_{n-1}^*(-l),
\mathbf{h}_{n-1},\boldsymbol{\gamma}_{n-1},
\mathbf{b}_{n-1},\boldsymbol{\mu}_{n-1})\notag\\
&\qquad\times\zeta\left(-l,h_n,\frac{b_n-b_{n-1}}{\gamma_n}\right)\gamma_n^l.\notag
\end{align}
Finally, applying Proposition \ref{prop_dC} to the right-hand side,
we arrive at the assertion of Theorem \ref{main_powercase}.




\verb??\\
{\bf Driss Essouabri}\\
Univ. Lyon,
UJM-Saint-Etienne,\\
CNRS, Institut Camille Jordan UMR 5208,\\
Facult\'e des Sciences et Techniques,\\       
23 rue du Docteur Paul Michelon,\\
F-42023, Saint-Etienne, France\\
{\it E-mail address}: driss.essouabri@univ-st-etienne.fr

\medskip

\verb??\\
{\bf Kohji Matsumoto}\\
Graduate School of Mathematics,\\
Nagoya University,\\
Furo-cho, Chikusa-ku,\\
Nagoya 464-8602, Japan\\
{\it E-mail address}: kohjimat@math.nagoya-u.ac.jp

 \end{document}